\documentclass{amsart}
\usepackage{amscd,amsmath,amsthm,amssymb}
\usepackage{pstricks}
\usepackage{latexsym,amsbsy,mathrsfs}
\usepackage{xy}
\usepackage{xypic}
\usepackage{pgf,tikz}

\newtheorem{thm}{Theorem}[section]
\newtheorem{lem}[thm]{Lemma}

\theoremstyle{definition}
\newtheorem{example}[thm]{Example}

\newtheorem{defn}[thm]{Definition}

\newtheorem{rem}[thm]{Remark}\newtheorem{ass}[thm]{Assumption}

\numberwithin{equation}{thm}

\begin{document}
\title[MUTATION PAIRS AND TRIANGULATED QUOTIENTS]
{MUTATION PAIRS AND TRIANGULATED QUOTIENTS}

\author{Zengqiang Lin and Minxiong Wang}
\address{ School of Mathematical sciences, Huaqiao University,
Quanzhou\quad 362021,  China.} \email{lzq134@163.com}
\address{ School of Mathematical sciences, Huaqiao University,
Quanzhou\quad 362021,  China.} \email{mxw@hqu.edu.cn}

\thanks{Supported by the National Nature Science Foundation of China (Grants No. 11101084, 11126331),
and the Natural Science Foundation of Fujian Province (Grants No. 2013J05009)}



\begin{abstract} We define mutation pair in a pseudo-triangulated category. We prove that under certain conditions, for a mutation pair in a pseudo-triangulated category, the corresponding quotient category carries a natural triangulated structure. This result unifies many previous constructions of quotient triangulated categories.
\end{abstract}

\maketitle

\section{Introduction}
Triangulated category was introduced by Grothendieck and later by Verdier in the sixties of last century. It is an important structure in both geometry and algebra. Quotient categories give a way to produce triangulated categories.

As shown by Happel \cite{[H]}, if we are given an exact category $(\mathcal{B},\mathcal{S})$ satisfying Frobenius condition, that is, $(\mathcal{B},\mathcal{S})$ has enough $\mathcal{S}$-injectives and enough $\mathcal{S}$-projectives and the $\mathcal{S}$-injectives coincide with the $\mathcal{S}$-projectives, then the quotient category $\mathcal{B}/\mathcal{I}$ carries a triangulated structure, where $\mathcal{I}$ is the full subcategory of $\mathcal{S}$-injectives. Beligiannis \cite{[B]} obtained a similar result by replacing $\mathcal{B}$ with a triangulated category $\mathcal{C}$ and replacing $\mathcal{S}$ with a proper class of triangles $\mathcal{E}$. Based on \cite{[BM]},  J{\o}rgensen \cite{[J]} showed that if $\mathcal{X}$ is a functorially finite subcategory of a triangulated category $\mathcal{C}$ with Auslander-Reiten translation $\tau$, and if $\tau\mathcal{X}=\mathcal{X}$, then the quotient $\mathcal{C}/\mathcal{X}$ is a triangulated category.

As a generalization of mutation of cluster tilting objects in cluster categories, Iyama and Yoshino \cite{[IY]} defined  mutation pair in a triangulated category. They showed that if $(\mathcal{Z},\mathcal{Z})$ is a $\mathcal{D}$-mutation pair in a triangulated category $\mathcal{C}$ and
$\mathcal{Z}$ is an extension-closed subcategory of $\mathcal{C}$, where $\mathcal{D}\subset\mathcal{Z}$ and $\mathcal{D}$ is rigid,
then the quotient $\mathcal{Z}/\mathcal{D}$ is a triangulated category. Recently Liu and Zhu \cite{[LZ]} defined  $\mathcal{D}$-mutation  pair in a right triangulated category,  and got a similar result, which unifies the constructions of Iyama-Yoshino in \cite{[IY]}
and J{\o}rgensen in \cite{[J]}.

To unify the constructions of quotient triangulated structures in \cite{[H]} and \cite{[IY]}, Nakaoka \cite{[N]} defined  Frobenius condition on a pseudo-triangulated category, which is similar to that on an exact category, and constructed a quotient triangulated category. Pseudo-triangulated categories are a natural   generalization of abelian categories and triangulated categories, whose right triangles and left triangles behave much better than those on pretriangulated categories. Unfortunetly  Nakaoka pointed out that his construction cannot cover  Beligiannis's result in \cite{[B]}.

The aim of this article is to give a way to unify the different constructions of quotient triangulated category. We define a $\mathcal{D}$-mutation pair $(\mathcal{Z},\mathcal{Z})$ in a pseudo-triangulated category $\mathcal{C}$, and prove that the quotient category  $\mathcal{Z}/\mathcal{D}$ carries a triangulated structure under certain reasonable conditions. As applications, Our result unifies the quotient triangulated categories considered by Iyama-Yoshino  \cite{[IY]}, by Happel  \cite{[H]}, by Beligiannis  \cite{[B]}, by J{\o}rgensen  \cite{[J]} and by Nakaoka \cite{[N]} respectively, but not \cite{[LZ]}.

The paper is organized as follows. In Section 2, we review the definition of pseudo-triangulated category and give some preliminaries. In Section 3, we define $\mathcal{D}$-mutation pair in a pseudo-triangulated category and prove our main result Theorem \ref{3d6}. In Section 4, we give five examples.

\section{pseudo-triangulated categories}

We recall some basics on pseudo-triangulated categories from \cite{[N]}.

 Let $\mathcal{C}$ be an additive category  and $\Sigma:\mathcal{C}\rightarrow\mathcal{C}$  an additive
functor. A sextuple $(A, B, C, f, g, h)$ in $\mathcal{C}$ is of the form  $A\xrightarrow{f}B\xrightarrow{g}C\xrightarrow{h}\Sigma A$. A morphism of sextuples from $A\xrightarrow{f}B\xrightarrow{g}C\xrightarrow{h}\Sigma A$ to $A'\xrightarrow{f'}B'\xrightarrow{g'}C'\xrightarrow{h'}\Sigma A'$ is a triple $(a,b,c)$ of morphisms such that the following diagram commutes:
$$\xymatrix{
A\ar[r]^{f}\ar[d]^{a} & B\ar[r]^{g}\ar[d]^{b} &
C\ar[r]^{h}\ar[d]^{c} & \Sigma A\ar[d]^{\Sigma a} \\
A'\ar[r]^{f'} & B'\ar[r]^{g'} & C'\ar[r]^{h'} & \Sigma A'. }
$$ If in addition $a,b$ and $c$ are isomorphisms in $\mathcal{C}$, then $(a,b,c)$ is called an isomorphism of sextuples.

\begin{defn} (cf. \cite{[BM]}, \cite{[N]}) Let $\mathcal{C}$ be an additive category  and $\Sigma:\mathcal{C}\rightarrow\mathcal{C}$  an additive
functor. Let $\rhd$ be a class of sextuples, whose elements are called $right\ triangles$. The triple $(\mathcal{C},\Sigma,\rhd)$ is called a $pseudo-right \ triangulated\ category$, and $\Sigma$ is called $suspension\ functor$, if  the following axioms are satisfied:

(RTR0) $\rhd$ is closed under isomorphisms.

(RTR1) For any $A\in\mathcal{C}$,
$0\xrightarrow{}A\xrightarrow{1_{A}}A\xrightarrow{}0$ is a right
triangle. And for any morphism $f:A\rightarrow B$ in $\mathcal{C}$,
there exists a right triangle
$A\xrightarrow{f}B\xrightarrow{g}C\xrightarrow{h}\Sigma A$.

(RTR2) If $A\xrightarrow{f}B\xrightarrow{g}C\xrightarrow{h}\Sigma A$ is a
right triangle, then so is $B\xrightarrow{g}C\xrightarrow{h}\Sigma
A\xrightarrow{-\sum f}\Sigma B$.

(RTR3) For any two right triangles
$A\xrightarrow{f}B\xrightarrow{g}C\xrightarrow{h}\Sigma A$ and
$A'\xrightarrow{f'}B'\xrightarrow{g'}C'\xrightarrow{h'}\Sigma A'$,
and any two morphisms $a:A\rightarrow A'$, $b:B\rightarrow B'$ such
that $bf=f'a$, there exists $c:C\rightarrow C'$ such that $(a,b,c)$ is a morphism of right triangles.

(RTR4)  Let $A\xrightarrow{f}B\xrightarrow{g}C\xrightarrow{h}\Sigma A$,
 $A\xrightarrow{l}M\xrightarrow{m}B'\xrightarrow{n}\Sigma A$ and
 $A'\xrightarrow{l'}M\xrightarrow{m'}B\xrightarrow{n'}\Sigma A'$
be right triangles satisfying $m'l=f$. Then there exist $g'\in\mathcal{C}(B',C)$ and $h'\in\mathcal{C}(C,\Sigma A')$ such that the following diagram is commutative and the third column is a right triangle.
$$\xymatrix{
 & A'\ar@{=}[r]\ar[d]^{l'} & A'\ar[d]^{f'}\\
 A\ar[r]^{l}\ar@{=}[d] & M \ar[r]^{m}\ar[d]^{m'} &
 B'\ar[r]^{n}\ar@{-->}[d]^{g'} & \Sigma A \ar@{=}[d]\\
 A\ar[r]^{f} & B\ar[r]^{g}\ar[d]^{n'} &
C\ar[r]^{h}\ar@{-->}[d]^{h'} & \Sigma A \ar[d]^{\Sigma l}\\
 & \Sigma A'\ar@{=}[r] & \Sigma A' \ar[r]^{-\Sigma l'} & \Sigma M} $$
\end{defn}

\begin{rem} In particular, if suspension functor $\Sigma$ is an equivalence, then $\mathcal{C}$ is a triangulated category. Pseudo-left triangulated categories $(\mathcal{C},\Omega, \triangleleft)$ are defined dually, where $\Omega:\mathcal{C}\rightarrow\mathcal{C}$ is called loop functor, and $\triangleleft$ is the class of left triangles.
\end{rem}

\begin{rem} Condition (RTR4) is slightly different from that in \cite{[BM]}. But the following two lemmas  are also true in this sense.
\end{rem}

\begin{lem} (\cite[Lemma 1.3]{[ABM]})
Let $\mathcal{C}$ be a pseudo-right triangulated category, and
$A\xrightarrow{f}B\xrightarrow{g}C\xrightarrow{h}\Sigma A$  a
right triangle.

(1) $g$ is a pseudocokernel of $f$, and $h$ is a pseudocokernel of
$g$.

(2) If $\Sigma$ is fully faithful, then $f$ is a pseudokernel of
$g$, and $g$ is a pseudokernel of $h$.
\end{lem}

\begin{lem}(\cite[Proposition 2.13]{[LZ]}) \label{2d4}
Let $(a,b,c)$ be a morphism of right triangles in a pseudo-right triangulated category $C$:
$$\xymatrix{
A\ar[r]^{f}\ar[d]^{a} & B\ar[r]^{g}\ar[d]^{b} &
C\ar[r]^{h}\ar[d]^{c} & \Sigma A\ar[d]^{\Sigma a} \\
A'\ar[r]^{f'} & B'\ar[r]^{g'} & C'\ar[r]^{h'} & \Sigma A'. }
$$ If $a$ and $b$ are isomorphisms, so is $c$.
\end{lem}

\begin{defn} (\cite[Definition 3.1]{[N]})
Let $(\mathcal{C}, \Sigma, \rhd)$ be a pseudo-right triangulated category, and $f\in\mathcal{C}(A,B)$.

(1) $f$ is called $\Sigma$-$null$ if it factors through some object in $\Sigma\mathcal{C}$.

(2) $f$ is called $\Sigma$-$epic$ if for any $b\in\mathcal{C}(B,B')$, $bf=0$ implies $b$ is $\Sigma$-null.
\end{defn}

For a pseudo-left triangulated category $(\mathcal{C}, \Omega, \lhd)$, we can define $\Omega$-null morphisms and $\Omega$-monic morphisms dually.

\begin{defn}(\cite[Definition 3.3]{[N]}) \
The sextuple $(\mathcal{C},\Sigma,\Omega,\rhd,\lhd,\psi)$ is called a $pseudo$-$triangulated\ category$ if $(\mathcal{C},\Sigma,\rhd)$ is a pseudo-right triangulated category, $(\mathcal{C},\Omega,\lhd)$ is a pseudo-left triangulated category, $(\Omega,\Sigma)$ is an adjoint pair with the adjugant $\psi:\hom (\Omega C, A)\xrightarrow{\sim}\hom (C,\Sigma A)$, which satisfy the following gluing conditions (G1) and (G2):

(G1) If $g\in\mathcal{C}(B,C)$ is $\Sigma$-epic, and $\Omega C\xrightarrow{e}A\xrightarrow{f}B\xrightarrow{g}C\in\lhd$, then $A\xrightarrow{f}B\xrightarrow{g}C\xrightarrow{-\psi(e)}\Sigma A\in\triangleright$.

(G2) If $f\in\mathcal{C}(A,B)$ is $\Omega$-monic, and $A\xrightarrow{f}B\xrightarrow{g}C\xrightarrow{h}\Sigma A\in\rhd$, then $\Omega C\xrightarrow{-\psi^{-1}(h)}A\xrightarrow{f}B\xrightarrow{g}C\in\triangleleft$.
\end{defn}
\begin{example}(\cite[Example 3.4]{[N]}) \
Let $(\mathcal{C},\Sigma,\Omega,\rhd,\lhd,\psi)$ be a pseudo-triangulated category.

(1) $\mathcal{C}$ is an abelian category if and only if $\Sigma=\Omega=0$.

(2) $\mathcal{C}$ is a triangulated category if and only if $\Sigma$ is the quasi-inverse of $\Omega$.
\end{example}

\begin{defn}(\cite[Definition 4.1]{[N]})
Let $(\mathcal{C},\Sigma,\Omega,\rhd,\lhd,\psi)$ be a pseudo-triangulated category. A sequence in $\mathcal{C}$
$$\Omega C\xrightarrow{e}A\xrightarrow{f}B\xrightarrow{g}C\xrightarrow{h}\Sigma A$$ is called an $extension$ if $A\xrightarrow{f}B\xrightarrow{g}C\xrightarrow{h}\Sigma A\in\rhd$, $\Omega C\xrightarrow{e}A\xrightarrow{f}B\xrightarrow{g}C\in\lhd$, and $h=-\psi(e)$.
\end{defn}

A morphism of extensions from $\Omega C\xrightarrow{e}A\xrightarrow{f}B\xrightarrow{g}C\xrightarrow{h}\Sigma A$ to $\Omega C'\xrightarrow{e'}A'\xrightarrow{f'}B'\xrightarrow{g'}C'\xrightarrow{h'}\Sigma A'$ is a triple $(a,b,c)$ such that the following diagram is commutative $$\xymatrix{
\Omega C\ar[r]^{e}\ar[d]^{\Omega c} & A\ar[r]^{f}\ar[d]^{a} & B\ar[r]^{g}\ar[d]^{b} &
C\ar[r]^{h}\ar[d]^{c} & \Sigma A\ar[d]^{\Sigma a} \\
\Omega C'\ar[r]^{e'} & A'\ar[r]^{f'} & B'\ar[r]^{g'} & C'\ar[r]^{h'} & \Sigma A'. }
$$ Note that $ae=e'\cdot \Omega c$ is equivalent to $\Sigma a\cdot h=h'c$, thus a morphism of extensions is the same as a morphism in $\rhd$ or $\lhd$.

\begin{example} (cf. \cite[Proposition 4.6]{[N]}) \
Let $\mathcal{C}$ be a pseudo-triangulated category.

(1) For any $A,B\in\mathcal{C}$, $\Omega B\xrightarrow{0}A\xrightarrow{i_A}A\oplus B\xrightarrow{p_B}B\xrightarrow{0}\Sigma A$ is an extension, where $i_A$ and $p_B$ are the injection and projection respectively.

(2) If $\mathcal{C}$ is abelian, then an extension is nothing other than a short exact sequence.

(3) If $\mathcal{C}$ is a triangulated category, then an extension is nothing other than a distinguished triangle.
\end{example}

The following lemma  will be frequently used in the next section.

\begin{lem}\label{2d2} (\cite[Proposition 4.7]{[N]}) \
Let $\Omega C\xrightarrow{e}A\xrightarrow{f}B\xrightarrow{g}C\xrightarrow{h}\Sigma A$,
 $\Omega B'\xrightarrow{k}A\xrightarrow{l}M\xrightarrow{m}B'\xrightarrow{n}\Sigma A$ and
$\Omega B\xrightarrow{k'}A'\xrightarrow{l'}M\xrightarrow{m'}B\xrightarrow{n'}\Sigma A'$
be extensions satisfying $m'l=f$. Then there exist $g'\in\mathcal{C}(B',C)$ and $h'\in\mathcal{C}(C,\Sigma A')$ such that the following diagram is commutative and the fourth column is an extension.
$$\xymatrix{
& & \Omega B\ar[r]^{\Omega g}\ar[d]^{k'} & \Omega C\ar[d]^{-\psi^{-1}(h')} \\
& & A'\ar@{=}[r]\ar[d]^{l'} & A'\ar[d]^{f'}\\
\Omega B'\ar[r]^{k}\ar[d] & A\ar[r]^{l}\ar@{=}[d] & M \ar[r]^{m}\ar[d]^{m'} &
 B'\ar[r]^{n}\ar@{-->}[d]^{g'} & \Sigma A \ar@{=}[d]\\
\Omega C \ar[r]^{e} & A\ar[r]^{f} & B\ar[r]^{g}\ar[d]^{n'} &
C\ar[r]^{h}\ar@{-->}[d]^{h'} & \Sigma A \ar[d]^{\Sigma l}\\
& & \Sigma A'\ar@{=}[r] & \Sigma A' \ar[r]^{-\Sigma l'} & \Sigma M} $$
\end{lem}

In the rest of this section, we give some properties on $\Sigma$-epic morphisms and $\Omega$-monic morphisms in a pseudo-triangulated category $\mathcal{C}$.

\begin{lem} \label{2d3}
 (1)Let $\Omega C\xrightarrow{e}A\xrightarrow{f}B\xrightarrow{g}C$ be a left triangle. Then the following statements are equivalent.

(a) $g$ is $\Sigma$-epic.

(b) $A\xrightarrow{f}B\xrightarrow{g}C\xrightarrow{-\psi(e)}\Sigma A$ is a right triangle.

(c) $\Omega C\xrightarrow{e}A\xrightarrow{f}B\xrightarrow{g}C\xrightarrow{-\psi(e)}\Sigma A$ is an extension.

(2) Let
$A\xrightarrow{f}B\xrightarrow{g}C\xrightarrow{h}\Sigma A$ be a right triangle. Then the following statements are equivalent.

(a) $f$ is $\Omega$-monic.

(b) $\Omega C\xrightarrow{-\psi^{-1}(h)}A\xrightarrow{f}B\xrightarrow{g}C$ is a left triangle.

(c) $\Omega C\xrightarrow{-\psi^{-1}(h)}A\xrightarrow{f}B\xrightarrow{g}C\xrightarrow{h}\Sigma A$ is an extension.

\end{lem}

\begin{proof}
We only show (1). (a)$\Rightarrow$ (b) follows from gluing condition (G1). By definition of extension we get (b)$\Leftrightarrow$(c). It remains to show (b)$\Rightarrow$ (a). Since $B\xrightarrow{g}C\xrightarrow{-\psi(e)}\Sigma A\xrightarrow{-\Sigma f}\Sigma B$ is a right triangle, $g$ is $\Sigma$-epic by definition.
\end{proof}

\begin{lem} \label{3d5}
 The following statements are equivalent.

(1) $f:A\rightarrow B$ is $\Sigma$-epic.

(2) For any right triangle $A\xrightarrow{f}B\xrightarrow{g}C\xrightarrow{h}\Sigma A$, there exists an object $C'\in\mathcal{C}$ such that $C\cong \Sigma C'$.

(3) There exists a right triangle $A\xrightarrow{f}B\xrightarrow{g}\Sigma C'\xrightarrow h\Sigma A$.
\end{lem}

\begin{proof}
(1)$\Rightarrow$ (2). Let $\Omega B\xrightarrow d C'\xrightarrow e A\xrightarrow{f}B$ be a left triangle. Since $f$ is $\Sigma$-epic,
$C'\xrightarrow e A\xrightarrow{f}B\xrightarrow{-\psi(d)}\Sigma C'$ is a right triangle. By (RTR2),
$A\xrightarrow{f}B\xrightarrow{-\psi(d)}\Sigma C'\xrightarrow{-\Sigma e}\Sigma A$ is a right triangle. So $C\cong \Sigma C'$ by Lemma \ref{2d4}.
(2)$\Rightarrow$ (3) and (3)$\Rightarrow$ (1) are clear.
\end{proof}

\begin{lem} \label{3d7}
Let $f:A\rightarrow B$, $g:B\rightarrow C$ and $h:A\rightarrow C$ be morphisms in $\mathcal{C}$ such that $h=gf$.

(1) If $h$ is $\Sigma$-epic, then so is $g$;

(2) If $h$ is $\Omega$-monic, then so is $f$;

(3) If $f$ and $g$ are $\Sigma$-epic, then so is $h$.

\end{lem}

\begin{proof} For (1) and (2) we see \cite[Lemma 4.4]{[N]}. Now prove (3).
Since $g:B\rightarrow C$ is $\Sigma$-epic, there exists a right triangle $L\xrightarrow{f'}B\xrightarrow{g} C\xrightarrow{h'}\Sigma L$ by Lemma \ref{2d3}.
By (RTR4), we get the following commutative diagram
$$\xymatrix{
 & L\ar@{=}[r]\ar[d]^{f'} & L\ar[d]^{a}\\
 A\ar[r]^{f}\ar@{=}[d] & B \ar[r]^{g'}\ar[d]^{g} &
 M\ar[r]^{h''}\ar[d]^{b} & \Sigma A \ar@{=}[d]\\
 A\ar[r]^{h} & C\ar[r]^{b'}\ar[d]^{h'} &
N\ar[r]^{c'}\ar[d]^{c} & \Sigma A \\
 & \Sigma L\ar@{=}[r] & \Sigma L} $$
 where the middle two columns and middle two rows are right triangles.
 Since $f:A\rightarrow B$ is $\Sigma$-epic, there exits an isomorphism $m:M\xrightarrow{\sim} \Sigma M'$ with $M'\in\mathcal{C}$ by Lemma \ref{3d5}. Thus $ma:L\rightarrow\Sigma M'$ is  $\Sigma$-epic by definition, and there exists a right triangle $L\xrightarrow{ma} \Sigma M'\xrightarrow{m'} \Sigma N'\xrightarrow{n'}\Sigma L$ by Lemma \ref{3d5}. By (RTR3) and Lemma \ref{2d4} there exists an isomorphism  $n:N\xrightarrow{\sim}\Sigma N'$ such that the following diagram  is commutative.
 $$\xymatrix{
L\ar[r]^{a}\ar@{=}[d] & M \ar[r]^{b}\ar[d]^{m} &
N\ar[r]^{c}\ar[d]^{n} & \Sigma L \ar@{=}[d]\\
 L\ar[r]^{ma} & \Sigma M'\ar[r]^{m'} &
\Sigma N'\ar[r]^{n'} & \Sigma L} $$
 According to Lemma \ref{3d5} again, $h$ is $\Sigma$-epic.
\end{proof}

\section{Main result}

Let $\mathcal{D}$ be a subcategory of an additive category $\mathcal{C}$. A morphism $f:A\rightarrow B$ in $\mathcal{C}$
is called $\mathcal{D}$-epic, if for any $D\in\mathcal{D}$, the sequence $\mathcal{C}(D,A)\xrightarrow{\mathcal{C}(D,f)}\mathcal{C}(D,B)\rightarrow 0$
is exact. A right $\mathcal{D}$-approximation of $X$ in
$\mathcal{C}$ is a $\mathcal{D}$-epic map $f: D\rightarrow X$, with $D$ in
$\mathcal{D}$. If for any object $X\in\mathcal{C}$, there exists a right
$\mathcal{D}$-approximation $f:D\rightarrow X$, then $\mathcal{D}$
is called a contravariantly finite subcategory of $\mathcal{C}$.
Dually we have the notions of $\mathcal{D}$-monic map, left $\mathcal{D}$-approximation and
covariantly finite subcategory. The subcategory $\mathcal{D}$ is called functorially
finite if $\mathcal{D}$ is both contravariantly finite and covariantly
finite.

\begin{defn}
 Let $\mathcal{D}$ and $\mathcal{Z}$ be two subcategories of a pseudo-triangulated category $\mathcal{C}$, and $\mathcal{D}\subseteq\mathcal{Z}$. The pair $(\mathcal{Z},\mathcal{Z})$ is called a $\mathcal{D}$-$mutation\ pair$ if it satisfies

(1) For any $X\in\mathcal{Z}$, there exists an extension $\Omega Y\xrightarrow{e}X\xrightarrow{f}D\xrightarrow{g}Y\xrightarrow{h}\Sigma X$ such that  $Y\in\mathcal{Z}$, $f$ is a left $\mathcal{D}$-approximation and $g$ is a right $\mathcal{D}$-approximation.

(2) For any $Y\in\mathcal{Z}$, there exists an extension $\Omega Y\xrightarrow{e}X\xrightarrow{f}D\xrightarrow{g}Y\xrightarrow{h}\Sigma X$ such that $X\in\mathcal{Z}$, $f$ is a left $\mathcal{D}$-approximation and $g$ is a right $\mathcal{D}$-approximation.
\end{defn}

\begin{defn}
Let $\mathcal{C}$ be a pseudo-triangulated category. A subcategory $\mathcal{Z}$ of $\mathcal{C}$ is called $extension$-$closed$ if for any extension in $\mathcal{C}$ $$\Omega Z\xrightarrow{e}X\xrightarrow{f}Y\xrightarrow{g}Z\xrightarrow{h}\Sigma X,$$ $X,Z\in\mathcal{Z}$ implies $Y\in\mathcal{Z}$.
\end{defn}

Let $\Omega Z\xrightarrow{e}X\xrightarrow{f}Y\xrightarrow{g}Z\xrightarrow{h}\Sigma X$ be an extension in $\mathcal{C}$ and $X,Y,Z\in\mathcal{Z}$, we simply say the extension is in $\mathcal{Z}$.

\begin{ass}\label{a1}
Assume that $\mathcal{C}$ is a  pseudo-triangulated category satisfying the following conditions:

Let $\Omega C\xrightarrow{e}A\xrightarrow{f}B\xrightarrow{g}C\xrightarrow{h}\Sigma A$ and $\Omega C'\xrightarrow{e'}A'\xrightarrow{f'}B'\xrightarrow{g'}C'\xrightarrow{h'}\Sigma A'$ be any two extensions in  $\mathcal{C}$.

(A1) If $c\in\mathcal{C}(C,C')$ satisfies $h'c=0$ and $cg=0$, then there exists $c'\in\mathcal{C}(C,B')$ such that $g'c'=c$.

(A2) If $a\in\mathcal{C}(A,A')$ satisfies $f'a=0$ and $ae=0$, then there exists $a'\in\mathcal{C}(B,A')$ such that $a'f=a$.


\end{ass}

\begin{rem}
If $\mathcal{C}$ is  a  triangulated category, it is easy to see that Assumption \ref{a1} is trivial. If $\mathcal{C}$ is an abelian category, then $g$ is epic so that $cg=0$ implies that $c=0$, thus we can take $c'=0$ in (A1). Therefore, Assumption \ref{a1} is trivial for both triangulated category and abelian category.
\end{rem}

From now on to the end of this section, we assume that $\mathcal{C}$ is a pseudo-triangulated category satisfying Assumption \ref{a1}, and assume that $(\mathcal{Z},\mathcal{Z})$ is a $\mathcal{D}$-mutation pair.

\begin{lem}\label{3d1}
Let $$\xymatrix{
\Omega Z\ar[r]^{e}\ar[d]^{\Omega z_i} & X\ar[r]^{f}\ar[d]^{x_i} & Y\ar[r]^{g}\ar[d]^{y_i} &
Z\ar[r]^{h}\ar[d]^{z_i} & \Sigma X\ar[d]^{\Sigma x_i \ \ \ \ \ (i=1,2)} \\
\Omega Z'\ar[r]^{e'} & X'\ar[r]^{f'} & D'\ar[r]^{g'} & Z'\ar[r]^{h'} & \Sigma X' }
$$ be morphisms of extensions in $\mathcal{Z}$, where $D'\in\mathcal{D}$ and  $f$ is  $\mathcal{D}$-monic, $i=1,2$. Then in quotient category $\mathcal{Z}/\mathcal{D}$,  $\underline{x_1}=\underline{x_2}$ implies that $\underline{z_1}=\underline{z_2}$.
\end{lem}

\begin{proof}
Since $\underline{x_1}=\underline{x_2}$, there exist $a_1:X\rightarrow D$, $a_2:D\rightarrow X'$ such that $x_1-x_2=a_2a_1$, where $D\in\mathcal{D}$.  Since $f$ is $\mathcal{D}$-monic, there exists $a_3:Y\rightarrow D$ such that $a_1=a_3f$. Thus $(x_1-x_2)e=a_2a_1e=a_2a_3fe=0$. Then $\Sigma (x_1-x_2)\cdot h=-\Sigma(x_1-x_2)(\psi(e))=-\psi((x_1-x_2)e)=0$. Note that $(y_1-y_2-f'a_2a_3)f=(y_1-y_2)f-f'(x_1-x_2)=0$, there exists $d:Z\rightarrow D'$ such that $y_1-y_2-f'a_2a_3=dg$. Now $(z_1-z_2-g'd)g=g'(y_1-y_2)-g'(y_1-y_2-f'a_2a_3)=0$, and $h'(z_1-z_2-g'd)=h'(z_1-z_2)=\Sigma (x_1-x_2)\cdot h=0$. By (A1), there exists $d':Z\rightarrow D'$ such that $z_1-z_2-g'd=g'd'$. So $z_1-z_2=g'(d+d')$ and $\underline{z_1}=\underline{z_2}$.
\end{proof}

\begin{lem}\label{3d2}
Let $\Omega Y\xrightarrow{e}X\xrightarrow{f}D\xrightarrow{g}Y\xrightarrow{h}\Sigma X$ and $\Omega Y'\xrightarrow{e'}X\xrightarrow{f'}D'\xrightarrow{g'}Y'\xrightarrow{h'}\Sigma X$ be extensions in $\mathcal{Z}$, where $f,f'$ are left $\mathcal{D}$-approximations. Then $Y$ and $Y'$ are isomorphic in $\mathcal{Z}/\mathcal{D}$.
\end{lem}

\begin{proof}
Since $f,f'$ are left $\mathcal{D}$-approximations, we obtain the following commutative diagram
$$\xymatrix{
\Omega Y\ar[r]^{e}\ar[d]^{\Omega y} & X\ar[r]^{f}\ar@{=}[d] & D\ar[r]^{g}\ar[d]^{d} &
Y\ar[r]^{h}\ar[d]^{y} & \Sigma X\ar@{=}[d] \\
\Omega Y'\ar[r]^{e'}\ar [d]^{\Omega y'} & X\ar[r]^{f'}\ar@{=}[d] & D'\ar[r]^{g'}\ar[d]^{d'} &
 Y'\ar[r]^{h'}\ar[d]^{ y'} & \Sigma X\ar@{=}[d] \\
\Omega Y\ar[r]^{e} & X\ar[r]^{f} & D\ar[r]^{g}  & Y\ar[r]^{h} & \Sigma X. }
$$ By Lemma \ref{3d1}, we get $\underline{y'} \underline{y}=\underline{1_{Y}}$. Similarly, we can show that $\underline{y} \underline{y'}=\underline{1_{Y'}}$. Hence $Y$ and $Y'$ are isomorphic in $\mathcal{Z}/\mathcal{D}$.
\end{proof}

For any $X\in\mathcal{Z}$, by definition of $\mathcal{D}$-mutation pair,  there exists an extension $\Omega TX\xrightarrow{e}X\xrightarrow{f}D\xrightarrow{g}TX\xrightarrow{h}\Sigma X$, where  $D\in\mathcal{D}$, $TX\in\mathcal{Z}$, $f$ is a left $\mathcal{D}$-approximation and $g$ is a right $\mathcal{D}$-approximation. By Lemma \ref{3d2}, $TX$ is unique up to isomorphism in the quotient category $\mathcal{Z}/\mathcal{D}$. So for any  $X\in\mathcal{Z}$, we fix an  extension as above. For any $x\in\mathcal{Z}(X,X')$, Since $f$ is a left $\mathcal{D}$-approximation, we can complete the following commutative diagram:
$$\xymatrix{
\Omega TX\ar[r]^{e}\ar@{-->}[d]^{\Omega y} & X\ar[r]^{f}\ar[d]^{x} & D\ar[r]^{g}\ar@{-->}[d]^{d} &
TX\ar[r]^{h}\ar@{-->}[d]^{y} & \Sigma X\ar[d]^{\Sigma x} \\
\Omega TX'\ar[r]^{e'} & X'\ar[r]^{f'} & D'\ar[r]^{g'} & TX'\ar[r]^{h'} & \Sigma X'. }
$$
We define a functor $T:\mathcal{Z}/\mathcal{D}\rightarrow\mathcal{Z}/\mathcal{D}$ by $T(X)=TX$ on the objects $X$ of $\mathcal{Z}/\mathcal{D}$ and by $T\underline{x}=\underline{y}$ on the morphisms $\underline{x}:X\rightarrow X'$ of $\mathcal{Z}/\mathcal{D}$. By Lemma \ref{3d1}, $T\underline{x}$ is well defined and $T$ is an additive functor. Sometimes we denote $y$ by $Tx$ for convenience.

\begin{lem}
The functor $T:\mathcal{Z}/\mathcal{D}\rightarrow\mathcal{Z}/\mathcal{D}$ is an equivalence.
\end{lem}

\begin{proof}
For any $Y\in\mathcal{Z}$, we fix an extension $\Omega Y\xrightarrow{e}T^{\ast}Y\xrightarrow{f}D\xrightarrow{g}Y\xrightarrow{h}\Sigma T^{\ast}Y$, where $D\in\mathcal{D}$, $T^{\ast}Y\in\mathcal{Z}$, $f$ is a left $\mathcal{D}$-approximation and $g$ is a right $\mathcal{D}$-approximation. We can similarly define an additive functor $T^{\ast}:\mathcal{Z}/\mathcal{D}\rightarrow\mathcal{Z}/\mathcal{D}$ by $T^{\ast}(Y)=T^{\ast}Y$. It is easy to check that $T^{\ast}T\cong id$ and $TT^{\ast}\cong id$. Thus $T$ is an equivalence.
\end{proof}

Let $\Omega Z\xrightarrow{u}X\xrightarrow{v}Y\xrightarrow{w}Z\xrightarrow{x}\Sigma X$ be an extension in $\mathcal{Z}$ and $v$ is $\mathcal{D}$-monic. Since $X\in\mathcal{Z}$, we have an extension $\Omega TX\xrightarrow{e}X\xrightarrow{f}D\xrightarrow{g}TX\xrightarrow{h}\Sigma X$. Then we can obtain the following commutative diagram since $v$ is $\mathcal{D}$-monic.
$$\xymatrix{
\Omega Z\ar[r]^{u}\ar[d]^{\Omega z} & X\ar[r]^{v}\ar@{=}[d] & Y\ar[r]^{w}\ar[d]^{y} &
Z\ar[r]^{x}\ar[d]^{z} & \Sigma X\ar@{=}[d] \\
\Omega TX\ar[r]^{e} & X\ar[r]^{f} & D\ar[r]^{g} & TX\ar[r]^{h} & \Sigma X }
$$
We call the sequence $X\xrightarrow{\underline{v}}Y\xrightarrow{\underline{w}} Z\xrightarrow{\underline{z}} TX$ a standard triangle in $\mathcal{Z}/\mathcal{D}$. We define the class of distinguished triangle $\triangle$ as the sextuples which are isomorphic to standard triangles.

\begin{lem}\label{3d4}
Let $$\xymatrix{
\Omega Z\ar[r]^{u}\ar[d]^{\Omega c} & X\ar[r]^{v}\ar[d]^{a} & Y\ar[r]^{w}\ar[d]^{b} &
Z\ar[r]^{x}\ar[d]^{c} & \Sigma X\ar[d]^{\Sigma a} \\
\Omega Z'\ar[r]^{u'} & X'\ar[r]^{v'} & Y'\ar[r]^{w'} & Z'\ar[r]^{x'} & \Sigma X' }
$$ be morphism of extensions in $\mathcal{Z}$, and $v,v'$ be $\mathcal{D}$-monic. Then we have the morphism of standard triangles in $\mathcal{Z}/\mathcal{D}$:
$$\xymatrix{
 X\ar[r]^{\underline{v}}\ar[d]^{\underline{a}} & Y\ar[r]^{\underline{w}}\ar[d]^{\underline{b}} &
Z\ar[r]^{\underline{z}}\ar[d]^{\underline{c}} & TX\ar[d]^{T\underline{a}} \\
 X'\ar[r]^{\underline{v}'} & Y'\ar[r]^{\underline{w}'} & Z'\ar[r]^{\underline{z}'} & TX'. }
$$
\end{lem}

\begin{proof} By the definition of standard triangle and definition of the functor $T$
 we have the following   commutative diagrams
$$\xymatrix{
\Omega Z\ar[r]^{u}\ar[d]^{\Omega z} & X\ar[r]^{v}\ar@{=}[d] & Y\ar[r]^{w}\ar[d]^{y} &
Z\ar[r]^{x}\ar[d]^{z} & \Sigma X\ar@{=}[d] \\
\Omega TX\ar[r]^{e}\ar[d]^{\Omega Ta} & X\ar[r]^{f}\ar[d]^{a} & D\ar[r]^{g}\ar[d]^{d} & TX\ar[r]^{h}\ar[d]^{Ta}& \Sigma X\ar[d]^{\Sigma a} \\
\Omega TX'\ar[r]^{e'} & X'\ar[r]^{f'} & D'\ar[r]^{g'} & TX'\ar[r]^{h'} & \Sigma X',}
$$and
$$\xymatrix{
\Omega Z\ar[r]^{u}\ar[d]^{\Omega c} & X\ar[r]^{v}\ar[d]^{a} & Y\ar[r]^{w}\ar[d]^{b} &
Z\ar[r]^{x}\ar[d]^{c} & \Sigma X\ar[d]^{\Sigma a} \\
\Omega Z'\ar[r]^{u'}\ar[d]^{\Omega z'} & X'\ar[r]^{v'}\ar@{=}[d] & Y'\ar[r]^{w'}\ar[d]^{y'} & Z'\ar[r]^{x'}\ar[d]^{z'}& \Sigma X'\ar@{=}[d] \\
\Omega TX'\ar[r]^{e'} & X'\ar[r]^{f'} & D'\ar[r]^{g'} & TX'\ar[r]^{h'} & \Sigma X'.}
$$ By Lemma \ref{3d1} we get $T\underline{a}\cdot\underline{z}=\underline{z}'\cdot\underline{c}$.
\end{proof}

\begin{lem}\label{3d3}
Let
$$\xymatrix{
& X'\ar@{=}[r]\ar[d]^{l'} & X'\ar[d]^{f'} &\\
X\ar[r]^{l}\ar@{=}[d] & M\ar[r]^{m}\ar[d]^{m'} & Y'
\\
X\ar[r]^{f} & Y &
}$$ be a commutative diagram,
where $m$ is the pseudo-cokernel of $l$ and $m'l'=0$. If $l'$ and $f$ are $\mathcal{D}$-monic, then $f'$ is also $\mathcal{D}$-monic.
\end{lem}

\begin{proof}
For any $a: X'\rightarrow D$, where $D\in\mathcal{D}$. We need to show that $a$ factors through $f'$. Since $l'$ is $\mathcal{D}$-monic,  there exists a morphism $b:M\rightarrow D$ such that $a=bl'$. Since $f$ is $\mathcal{D}$-monic, there exists a morphism $c: Y\rightarrow D$ such that $bl=cf$. Thus $(b-cm')l=bl-cf=0$. There exists a morphism $d:Y'\rightarrow D$ such that $b-cm'=dm$. Hence $a=bl'=bl'-cm'l'=dml'=df'$.
\end{proof}

\begin{thm}\label{3d6}

Let $\mathcal{C}$ be a pseudo-triangulated category satisfying Assumption \ref{a1}, and $\mathcal{Z}$ be an extension-closed subcategory of $\mathcal{C}$. If $(\mathcal{Z},\mathcal{Z})$ is a $\mathcal{D}$-mutation pair , then  $(\mathcal{Z}/\mathcal{D}, T, \bigtriangleup )$ is a triangulated category.
\end{thm}

\begin{proof}
We will check that the distinguished triangles in $\triangle$ satisfy the axioms of triangulated category.

(TR1) Let $v\in \mathcal{Z}(X,Y)$. Take $v'=(v,f)^T: X\rightarrow Y\oplus D$ and $p_D=(0,1_D):Y\oplus D\rightarrow D$, where $f: X\rightarrow D$ is a left $\mathcal{D}$-approximation of $X$. Then $p_Dv'=f$. Since $f$ is $\mathcal{D}$-monic and $\Omega$-monic, we get $v'$ is $\mathcal{D}$-monic, and $v'$ is also $\Omega$-monic by Lemma \ref{3d7}(2). Now there are three extensions $\Omega TX\xrightarrow{e}X\xrightarrow{f}D\xrightarrow{g}TX\xrightarrow{h}\Sigma X$,
 $\Omega D\xrightarrow{0}Y\xrightarrow{i_Y}Y\oplus D\xrightarrow{p_D}D\xrightarrow{0}\Sigma Y$ and
 $\Omega Z\xrightarrow{u'}X\xrightarrow{v'}Y\oplus D\xrightarrow{w'}Z\xrightarrow{x'}\Sigma X$.
 By Lemma \ref{2d2}, there exist $g'\in\mathcal{C}(Z,TX)$ and $h'\in\mathcal{C}(TX,\Sigma Y)$ such that the following diagram is commutative and the fourth column is an extension.
$$\xymatrix{
& & \Omega D\ar[r]^{\Omega g}\ar[d]^{0} & \Omega TX\ar[d]^{-\psi^{-1}(h')} \\
& & Y\ar@{=}[r]\ar[d]^{i_Y} & Y\ar[d]^{f'}\\
\Omega Z\ar[r]^{u'}\ar[d] & X\ar[r]^{v'}\ar@{=}[d] & Y\oplus D \ar[r]^{w'}\ar[d]^{p_D} &
 Z\ar[r]^{x'}\ar@{-->}[d]^{g'} & \Sigma X \ar@{=}[d]& (3.1)\\
\Omega TX \ar[r]^{e} & X\ar[r]^{f} & D\ar[r]^{g}\ar[d]^{0} &
TX\ar[r]^{h}\ar@{-->}[d]^{h'} & \Sigma X \ar[d]^{\Sigma v'}\\
& & \Sigma Y\ar@{=}[r] & \Sigma Y \ar[r]^{-\Sigma i_Y} & \Sigma (Y\oplus D)} $$
Since $Y, TX\in\mathcal{Z}$ and $\mathcal{Z}$ is extension-closed, we get $Z\in\mathcal{Z}$. Thus
$X\xrightarrow{\underline{v'}}Y\oplus D \xrightarrow{\underline{w'}} Z\xrightarrow{\underline{g'}}TX$ is a standard triangle. So  $X\xrightarrow{\underline{v}}Y \xrightarrow{\underline{f'}} Z\xrightarrow{\underline{g'}}TX$ is a distinguished triangle. It is  easy to see that $\Omega 0\rightarrow X\xrightarrow{1_X}X \rightarrow 0\rightarrow \Sigma X$ is an extension and $1_X$ is $\mathcal{D}$-monic. Thus $X\xrightarrow{\underline{1}_X}X \rightarrow 0\rightarrow TX\in\Delta$.

(TR2) Let $$X\xrightarrow{\underline{v}}Y \xrightarrow{\underline{w}} Z'\xrightarrow{\underline{z}}TX\ \ \ \ \ \ (3.2) \ $$ be a distinguished triangle induced by the extension $\Omega Z'\xrightarrow{u}X\xrightarrow{v}Y\xrightarrow{w}Z'\xrightarrow{x}\Sigma X$, where $v$ is $\mathcal{D}$-monic.
By (TR1) we suppose that
$$X\xrightarrow{\underline{v}}Y \xrightarrow{\underline{f'}} Z\xrightarrow{\underline{g'}}TX\ \ \ \ \ \ (3.3) \ $$
is a distinguished triangle induced by $\underline{v}$. Since $v$ is $\mathcal{D}$-monic, there exists $y:Y\rightarrow D$ such that $f=yv$. Thus we have the following commutative diagram
$$\xymatrix{
\Omega Z'\ar[r]^{u}\ar[d]^{\Omega z} & X\ar[r]^{v}\ar@{=}[d] & Y\ar[r]^{w}\ar[d]^{(1_Y,y)^T} &
Z'\ar[r]^{x}\ar[d]^{z} & \Sigma X\ar@{=}[d] \\
\Omega Z\ar[r]^{u'}\ar[d]^{\Omega z'} & X\ar[r]^{v'}\ar@{=}[d] & Y\oplus D\ar[r]^{w'}\ar[d]^{p_Y} & Z\ar[r]^{x'}\ar[d]^{z'}& \Sigma X\ar@{=}[d] \\
\Omega Z'\ar[r]^{u} & X\ar[r]^{v} & Y\ar[r]^{w} & Z'\ar[r]^{x} & \Sigma X.}$$
Since $p_Y(1_Y,p)^T=1_Y$, we get $z'z$ is an isomorphism by Lemma \ref{2d4}. On the other hand, let $w'=(f',a')$. Since $w'$ is a pseudocokernel of $v'$ and  $(y,-1_D)v'=(y,-1_D)(v,f)^T=0$, there exists $c:Z\rightarrow D$ such that $(y,-1_D)=cw'=c(f',a')$. Thus $y=cf', ca'=-1_D$.
By Lemma \ref{2d2}, we have the following commutative diagram
$$\xymatrix{
& & \Omega Y\ar[r]^{\Omega w}\ar[d]^{0} & \Omega Z'\ar[d]^{-\psi^{-1}(b')} \\
& & D\ar@{=}[r]\ar[d]^{i_D} & D\ar[d]^{a'}\\
\Omega Z\ar[r]^{u'}\ar[d] & X\ar[r]^{v'}\ar@{=}[d] & Y\oplus D \ar[r]^{w'}\ar[d]^{p_Y} &
 Z\ar[r]^{x'}\ar[d]^{z'} & \Sigma X \ar@{=}[d]& \\
\Omega Z' \ar[r]^{u} & X\ar[r]^{v} & Y\ar[r]^{w}\ar[d]^{0} &
Z'\ar[r]^{x}\ar[d]^{b'} & \Sigma X \ar[d]^{\Sigma v'}\\
& & \Sigma D\ar@{=}[r] & \Sigma D \ar[r]^{-\Sigma i_D} & \Sigma (Y\oplus D),} $$
where the fourth column is an extension. Note that $z'$ is  a pseudocokernel of $a'$ and $(1_Z+a'c)a'=a'-a'=0$,
there exists $z'':Z'\rightarrow Z$ such that $z''z'=1_Z+a'c$.
So $\underline{z''z'}=\underline{1_Z}$. Therefore, $\underline{z}'$ is an isomorphism in $\mathcal{Z}/\mathcal{D}$,
and  the distinguished triangle (3.2) and (3.3) are isomorphic by Lemma \ref{3d4}. Now we can replace the triangle (3.2) with (3.3).
According to the proof of (TR1), we get an extension $\Omega TX\xrightarrow{-\psi^{-1}(h')}Y\xrightarrow{f'} Z\xrightarrow{g'}TX\xrightarrow{h'}\Sigma Y$ in $\mathcal{Z}$. Since $i_Y$ and $f$ are $\mathcal{D}$-monic, so is $f'$ by Lemma \ref{3d3}. Take an extension  $\Omega TY\xrightarrow{e_{Y}}Y\xrightarrow{f_Y}D_Y\xrightarrow{g_Y}TY\xrightarrow{h_Y}\Sigma Y$, where $f_Y$ is a left $\mathcal{D}$-approximation and $g_Y$ is a right $\mathcal{D}$-approximation. Then we get the following commutative diagram
$$\xymatrix{
\Omega TX\ar[r]^{-\psi^{-1}(h')}\ar[d]^{\Omega z'} & Y\ar[r]^{f'}\ar@{=}[d] & Z\ar[r]^{g'}\ar[d]^{y'} &
TX\ar[r]^{h'}\ar[d]^{z'} & \Sigma Y\ar@{=}[d] \\
\Omega TY\ar[r]^{e_Y} & Y\ar[r]^{f_Y} & D_Y\ar[r]^{g_Y} & TY\ar[r]^{h_Y} & \Sigma Y. }
$$ Thus $Y \xrightarrow{\underline{f'}} Z\xrightarrow{\underline{g'}}TX\xrightarrow{\underline{z}'}TY$ is a distinguished triangle.
To end the proof of (TR2) we only need to show that $\underline{z}'=-T\underline{v}$. The commutative diagram (3.1) implies the following commutative diagram
$$\xymatrix{
\Omega TX\ar[r]^{e}\ar[d]^{-1_{\Omega TX}} & X\ar[r]^{f}\ar[d]^{v} & D\ar[r]^{g}\ar[d]^{-a'} &
TX\ar[r]^{h}\ar[d]^{-1_{TX}} & \Sigma X\ar[d]^{\Sigma v} \\
\Omega TX\ar[r]^{-\psi^{-1}(h')} & Y\ar[r]^{f'} & Z\ar[r]^{g'} & TX\ar[r]^{h'} & \Sigma Y. }
$$ In fact, $-a'f=-w'i_{D}f=-w'i_{D}p_{D}v'=-w'(v'-i_Yp_Yv')=w'i_Yp_Yv'=f'v$, $g'a'=g'w'i_D=gp_Di_D=g$ and $\Sigma v\cdot h=\Sigma p_Y\cdot\Sigma v'\cdot h=\Sigma p_Y\cdot-\Sigma i_Y\cdot h'=-h'$.
Composing the above two commutative diagrams,  we obtain the following commutative diagram
$$\xymatrix{
\Omega TX\ar[r]^{e}\ar[d]^{-\Omega z'} & X\ar[r]^{f}\ar[d]^{v} & D\ar[r]^{g}\ar[d]^{-y'a'} &
TX\ar[r]^{h}\ar[d]^{-z'} & \Sigma X\ar[d]^{\Sigma v} \\
\Omega TY\ar[r]^{e_Y} & Y\ar[r]^{f_Y} & D_Y\ar[r]^{g_Y} & TY\ar[r]^{h_Y} & \Sigma Y, }
$$ which implies that $T\underline{v}=-\underline{z}'$.

(TR3) Suppose there is a commutative diagram
$$\xymatrix{
X\ar[r]^{\underline{v}}\ar[d]^{\underline{a}} &  Y\ar[r]^{\underline{w}}\ar[d]^{\underline{b}} &
Z\ar[r]^{\underline{z}} & TX\ar[d]^{T\underline{a}} \\
X'\ar[r]^{\underline{v}'} & Y'\ar[r]^{\underline{w}'} & Z'\ar[r]^{\underline{z}'} & TX', }$$
where the rows are distinguished triangles. We may assume that $v,v'$ are $\mathcal{D}$-monic and the distinguished triangles are arising from the extensions $\Omega Z\xrightarrow{u}X\xrightarrow{v}Y\xrightarrow{w}Z\xrightarrow{x}\Sigma X$ and $\Omega Z'\xrightarrow{u'}X'\xrightarrow{v'}Y'\xrightarrow{w'}Z'\xrightarrow{x'}\Sigma X'$ respectively. Since $\underline{b}\underline{v}=\underline{v}'\underline{a}$, there exist $a_1:X\rightarrow D, a_2:D\rightarrow Y'$ such that $bv-v'a=a_2a_1$, where $D \in \mathcal{D}$. Since $v$ is $\mathcal{D}$-monic, there exists $a_3:Y\rightarrow D$ such that $a_1=a_3v$. Thus $(b-a_2a_3)v=bv-a_{2}a_{1}=v'a$. So by (RTR3) there exists $c:Z\rightarrow Z'$ such that the following diagram is commutative
$$\xymatrix{
\Omega Z\ar[r]^{u}\ar[d]^{\Omega c} & X\ar[r]^{v}\ar[d]^{a} &  Y\ar[r]^{v}\ar[d]^{b-a_2a_3} &
Z\ar[r]^{x}\ar@{-->}[d]^{c} & \Sigma X\ar[d]^{\Sigma a} \\
\Omega Z'\ar[r]^{u'} & X'\ar[r]^{v'} & Y'\ar[r]^{w'} & Z'\ar[r]^{x'} & \Sigma X'. }$$
Hence (TR3) follows from Lemma \ref{3d4}.

(TR4) Let $X\xrightarrow{\underline{v}}Y \xrightarrow{\underline{w}} Z\xrightarrow{\underline{z}}TX$,
 $X'\xrightarrow{\underline{v'}}Y \xrightarrow{\underline{w'}} Z'\xrightarrow{\underline{z'}}TX'$ and
  $X\xrightarrow{\underline{w'v}}Z' \xrightarrow{\underline{q}} Y'\xrightarrow{\underline{t}}TX$  be distinguished triangles. Assume that $v$ and $v'$ are $\mathcal{D}$-monic. Let $f:X\rightarrow D$ be a left $\mathcal{D}$-approximation of $X$. Since
  $\left(
     \begin{array}{cc}
       w' & 0 \\
       0 & 1_D \\
     \end{array}
   \right)\left(
            \begin{array}{c}
              v \\
              f \\
            \end{array}
          \right)
  =\left(
     \begin{array}{c}
       w'v \\
       f \\
     \end{array}
   \right)
  $ is $\mathcal{D}$-monic, replacing $v$ by $\left(
            \begin{array}{c}
              v \\
              f \\
            \end{array}
          \right)$, $v'$ by $\left(
            \begin{array}{c}
              v' \\
              0 \\
            \end{array}
          \right)$, $w'$ by $\left(
     \begin{array}{cc}
       w' & 0 \\
       0 & 1_D \\
     \end{array}
   \right)$, and adjusting the corresponding  distinguished triangles, we can assume that $w'v$ is $\mathcal{D}$-monic too. Thus we can assume that the above three distinguished triangles are standard, which are induced by the following three extensions
   $\Omega Z\xrightarrow{u}X\xrightarrow{v}Y\xrightarrow{w}Z\xrightarrow{x}\Sigma X$, $\Omega Z'\xrightarrow{u'}X'\xrightarrow{v'}Y'\xrightarrow{w'}Z'\xrightarrow{x'}\Sigma X'$ and $\Omega Y'\xrightarrow{n}X\xrightarrow{w'v }Z'\xrightarrow{q}Y'\xrightarrow{r}\Sigma X$. By Lemma \ref{2d2}, we get the following commutative diagram
$$\xymatrix{
& & \Omega Z'\ar[r]^{\Omega q}\ar[d]^{u'} & \Omega Y'\ar[d]^{n'} \\
& & X'\ar@{=}[r]\ar[d]^{v'} & X'\ar[d]^{p'}\\
\Omega Z\ar[r]^{u}\ar[d]^{\Omega q'} & X\ar[r]^{v}\ar@{=}[d] & Y \ar[r]^{w}\ar[d]^{w'} &
 Z\ar[r]^{x}\ar@{-->}[d]^{q'} & \Sigma X \ar@{=}[d]\\
\Omega Y' \ar[r]^{n} & X\ar[r]^{w'v} & Z'\ar[r]^{q}\ar[d]^{x'} &
Y'\ar[r]^{r}\ar@{-->}[d]^{r'} & \Sigma X \ar[d]^{\Sigma v}\\
& & \Sigma X'\ar@{=}[r] & \Sigma X' \ar[r]^{-\Sigma v'} & \Sigma Y,} $$
where the third row, the fourth row, the third column and the fourth column are all extensions, moreover, $v,w'v,v',p'$ are all $\mathcal{D}$-monic by Lemma \ref{3d3}. Thus by Lemma \ref{3d4} we get the following commutative diagram of distinguished triangles:
 $$\xymatrix{
 & X'\ar@{=}[r]\ar[d]^{\underline{v}'} & X'\ar[d]^{\underline{p}'}\\
X\ar[r]^{\underline{v}}\ar@{=}[d] & Y \ar[r]^{\underline{w}}\ar[d]^{\underline{w}'} &
 Z\ar[r]^{\underline{z}}\ar[d]^{\underline{q}'} & TX \ar@{=}[d]\\
 X\ar[r]^{\underline{w}'\underline{v}} & Z'\ar[r]^{\underline{q}}\ar[d]^{\underline{z}'} &
Y'\ar[r]^{\underline{t}}\ar[d]^{\underline{t}'} & TX \\
 & TX'\ar@{=}[r] & TX'. } $$
To end the proof, it remains to show that the following diagram is commutative:
$$\CD
  Y' @>\underline{t}>> TX\\
  @V \underline{t}' VV @V T\underline{v} VV \ \ \ \  (3.4) \\
  TX' @>-T\underline{v}'>> TY
\endCD$$

We first claim that there exist morphisms of extensions
$$\xymatrix{
\Omega Y'\ar[r]^{n}\ar[d]^{\Omega b} & X\ar[r]^{w'v}\ar[d]^{v} &  Z'\ar[r]^{q}\ar[d]^{a} &
Y'\ar[r]^{r}\ar[d]^{b} & \Sigma X\ar[d]^{\Sigma v \ \ \ \ \mbox{(3.5)} }\\
\Omega TY\ar[r]^{e'} & Y\ar[r]^{f'} & D'\ar[r]^{g'} & TY\ar[r]^{h'} & \Sigma Y }$$
and $$\xymatrix{
\Omega Y'\ar[r]^{n'}\ar[d]^{\Omega b'} & X'\ar[r]^{p'}\ar[d]^{v'} &  Z\ar[r]^{q'}\ar[d]^{a'} &
Y'\ar[r]^{r'}\ar[d]^{b'} & \Sigma X'\ar[d]^{\Sigma v'} \\
\Omega TY\ar[r]^{e'} & Y\ar[r]^{f'} & D'\ar[r]^{g'} & TY\ar[r]^{h'} & \Sigma Y, }$$such that $f'=aw'+a'w$, $\underline{b}=T\underline{v}\cdot \underline{t}$ and $\underline{b'}=T\underline{v'}\cdot \underline{t'}$.

In fact, since $w'v$ is $\mathcal{D}$-monic, there exists $a:Z'\rightarrow D'$ such that $f'v=aw'v$, then by (RTR3) there exists $b:Y'\rightarrow TY$ such that $(v,a,b)$ is a morphism of extensions. Note that $(f'-aw')v=0$, there exists $a':Z\rightarrow D'$ such that $a'w=f'-aw'$. Thus $f'=a'w+aw'$ and $f'v'=a'wv'+aw'v'=a'p'$, then  by (RTR3) there exists $b':Y'\rightarrow TY$ such that $(v',a',b')$ is a morphism of extensions. By the construction of standard triangle, we have the following commutative diagram
$$\xymatrix{
\Omega Y'\ar[r]^{n}\ar[d]^{\Omega t} & X\ar[r]^{w'v}\ar@{=}[d] &  Z'\ar[r]^{q}\ar[d]^{s} &
Y'\ar[r]^{r}\ar[d]^{t} & \Sigma X\ar@{=}[d] \\
\Omega TX\ar[r]^{e} & X\ar[r]^{f} & D\ar[r]^{g} & TX\ar[r]^{h} & \Sigma X.}$$
On the other hand, by the definition of $Tv$, we have the following commutative diagram
$$\xymatrix{
\Omega TX\ar[r]^{e}\ar[d]^{\Omega Tv} & X\ar[r]^{f}\ar[d]^{v} & D\ar[r]^{g}\ar[d]^{d} &
TX\ar[r]^{h}\ar[d]^{Tv} & \Sigma X\ar[d]^{\Sigma v} \\
\Omega TY\ar[r]^{e'} & Y\ar[r]^{f'} & D'\ar[r]^{g'} & TY\ar[r]^{h'} & \Sigma Y. }
$$ Composing  the last two diagrams, we immediately obtain the following commutative diagram
$$\xymatrix{
\Omega Y'\ar[r]^{n}\ar[d]^{\Omega (Tv\cdot t)} & X\ar[r]^{w'v}\ar[d]^{v} &  Z'\ar[r]^{q}\ar[d]^{ds} &
Y'\ar[r]^{r}\ar[d]^{Tv\cdot t} & \Sigma X\ar[d]^{\Sigma v \ \ \ \ \mbox{(3.6)}}  \\
\Omega TY\ar[r]^{e'} & Y\ar[r]^{f'} & D'\ar[r]^{g'} & TY\ar[r]^{h'} & \Sigma Y.}$$
Comparing diagram (3.5) and diagram (3.6), we get $\underline{b}=T\underline{v}\cdot \underline{t}$ by Lemma \ref{3d1}. Similarly we can show $\underline{b'}=T\underline{v'}\cdot \underline{t'}$.

Now we can show  diagram (3.4) is commutative. On one hand, $h'(b+b')=\Sigma v\cdot r+\Sigma v'\cdot r'=0$. On the other hand, $(b+b')q'w=bqw'+b'q'w=g'aw'+g'a'w=g'f'=0$. Since $w,q'$ are $\Sigma$-epic, so is $q'w$ by Lemma \ref{3d7}(3). Thus we obtain an extension $\Omega Y'\xrightarrow{\alpha} Y''\xrightarrow {\beta}Y\xrightarrow{q'w} Y'\xrightarrow{\gamma}\Sigma Y''$ by Lemma \ref{2d3}. So $b+b'$ factors through $D'$ by (A1). Then $T\underline{v}\cdot \underline{t}+T\underline{v'}\cdot \underline{t'}=\underline{b}+\underline{b}'=0$.
\end{proof}

\begin{rem}
In the proof of (TR4), we see \cite{[C]} for the reason why we can adjust the morphisms $v$ and $w'$.
\end{rem}

\section{Examples}

Since Assumption \ref{a1} is trivial for both triangulated categories and abelian categories,   we can apply Theorem \ref{3d6} to several situations.

\begin{example}(\cite[Theorem 4.2]{[IY]})
Let $\mathcal{D}$ be a rigid subcategory of a triangulated category $\mathcal{C}$ and $\mathcal{Z}$ be a subcategory of $\mathcal{C}$ such that $\mathcal{D}\subseteq \mathcal{Z}$. We call $(\mathcal{Z},\mathcal{Z})$  a $\mathcal{D}$-mutation pair in the sense of Iyama-Yoshino, if it satisfies

 (1) $\mathcal{C}(\mathcal{Z},\Sigma\mathcal{D})=\mathcal{C}(\mathcal{D},\Sigma\mathcal{Z})=0$.

 (2) For any $X\in\mathcal{Z}$, there exists a distinguished triangle $X\xrightarrow{f} D\xrightarrow{g} Y\xrightarrow{h} \Sigma X$ with $D\in\mathcal{D}$ and $Y\in\mathcal{Z}$.

 (3) For any $Y\in\mathcal{Z}$, there exists a distinguished triangle $X\xrightarrow{f} D\xrightarrow{g} Y\xrightarrow{h} \Sigma X$ with $D\in\mathcal{D}$ and $X\in\mathcal{Z}$.

In condition (2) and (3), $f$ is a left $\mathcal{D}$-approximation and $g$ is a right $\mathcal{D}$-approximation immediately follows from condition (1). Thus $(\mathcal{Z},\mathcal{Z})$ is a $\mathcal{D}$-mutation pair in the sense of Definition 3.1. In this case, if  $\mathcal{Z}$ is extension-closed, then $\mathcal{Z}/\mathcal{D}$ is a triangulated category by Theorem \ref{3d6}.
\end{example}

\begin{example}(\cite[Theorem 2.6]{[H]})
Let $\mathcal{C}$ be an abelian category. Let $(\mathcal{B},\mathcal{S})$ be a Frobenius subcategory of $\mathcal{C}$ and $\mathcal{I}$  the subcategory of $\mathcal{B}$ consisting of all $S$-injectives. Then $\mathcal{B}/\mathcal{I}$ is a triangulated category.
\end{example}

\begin{proof}
According to Theorem \ref{3d6}, we need to show that $(\mathcal{B},\mathcal{B})$ is a $\mathcal{I}$-mutation pair. In fact, for any $X\in\mathcal{B}$, since $\mathcal{B}$ has enough $\mathcal{S}$-injectives, there exists a short exact sequence $0\rightarrow X\xrightarrow{f} I\xrightarrow{g} Y\rightarrow 0$ in $\mathcal{S}$, where $I\in\mathcal{I}$. It is clear that $f$ is a left $\mathcal{I}$-approximation of $X$. Since  the $\mathcal{S}$-projectives coincide with the $\mathcal{S}$-injectives, $g$ is a right $\mathcal{I}$-approximation of $Y$ by definition. The second condition  can be showed similarly.
\end{proof}

\begin{example}(\cite[Theorem 7.2]{[B]})
Let $\mathcal{C}$ be a triangulated category and $\mathcal{E}$  a proper class of triangles on $\mathcal{C}$. An object $I\in\mathcal{C}$ is called $\mathcal{E}$-injective, if for any distinguished triangle $A\rightarrow B\rightarrow C\rightarrow \Sigma A$ in $\mathcal{E}$, the induced sequence $0\rightarrow \mathcal{C}(C,I)\rightarrow \mathcal{C}(B,I)\rightarrow \mathcal{C}(A,I)\rightarrow 0$ is exact. Denote by $\mathcal{I}$ the full subcategory of $\mathcal{E}$-injective objects of $\mathcal{C}$. We say that $\mathcal{C}$ has enough $\mathcal{E}$-injectives if for any object $A\in\mathcal{C}$ there exists a distinguished triangle $A\rightarrow I\rightarrow C\rightarrow \Sigma A$ in $\mathcal{C}$ with $I\in\mathcal{I}$.  If $\mathcal{C}$ has enough $\mathcal{E}$-injectives and enough $\mathcal{E}$-projectives and $\mathcal{I}=\mathcal{P}$, where $\mathcal{P}$ is the subcategory of $\mathcal{E}$-projectives, then we can show that $(\mathcal{C},\mathcal{C})$ is a $\mathcal{I}$-mutation pair, thus $\mathcal{C}/\mathcal{I}$ is a triangulated category by Theorem \ref{3d6}. We remark that $\mathcal{C}(\mathcal{I},\Sigma\mathcal{I})$ is not zero because $\mathcal{I}$ is closed under $\Sigma$. So $(\mathcal{C},\mathcal{C})$ is not a $\mathcal{I}$-mutation pair in the sense of Iyama-Yoshino.
\end{example}

\begin{example}(\cite[Theorem 2.3]{[J]})
Let $\mathcal{C}$ be a triangulated category with Auslander-Reiten translation $\tau$, and $\mathcal{X}$ a functorially finite subcategory with $\tau\mathcal{X}=\mathcal{X}$. Let $X\xrightarrow{f} D\xrightarrow{g} Y\xrightarrow{h} \Sigma X$ be a triangle in $\mathcal{C}$, then we can show that $f$ is a left $\mathcal{X}$-approximation if and only if $g$ is a right $\mathcal{X}$-approximation (see \cite[Lemma 2.2]{[J]}). Thus $(\mathcal{C},\mathcal{C})$ is a $\mathcal{X}$-mutation pair, and $\mathcal{C}/\mathcal{X}$ is a triangulated category by Theorem \ref{3d6}. We remark that $(\mathcal{C},\mathcal{C})$ may be not a $\mathcal{X}$-mutation pair in the sense of Iyama-Yoshino. For example, if $\mathcal{C}$ is a cluster category, then $\mathcal{C}(\mathcal{X},\Sigma\mathcal{X})$ is not zero.
\end{example}

\begin{example}(\cite[Theorem 6.17]{[N]})
Let $\mathcal{C}$ be a pseudo-triangulated category, if $(\mathcal{C},\mathcal{Z},\mathcal{D})$ is Frobenius (see \cite[Definition 5.9]{[N]}), then $(\mathcal{Z},\mathcal{Z})$ is a $\mathcal{I_{D}}$-mutation pair, where $\mathcal{I_D}$ is the full subcategory of $\mathcal{D}$ consisting of injective objects. If  $\mathcal{Z}$ is extension-closed and $\mathcal{C}$ satisfies Assumption \ref{a1}, then $\mathcal{Z}/\mathcal{I_{D}}$ is a triangulated category by Theorem \ref{3d6}.
\end{example}

\vspace{2mm}\noindent {\bf Acknowledgements} \  The authors  thank Professor Xiaowu Chen for useful comments and advices.

\end{document}